\documentclass[12pt]{amsart}
\usepackage{amsmath}
\usepackage{amssymb}
\usepackage{amsfonts}
\usepackage{amsthm}
\usepackage{enumerate}
\usepackage[mathscr]{eucal}
\usepackage{verbatim}
\usepackage{amsthm}
\usepackage{amscd}

\newtheorem{theorem}{Theorem}[section]

\newtheorem{lemma}[theorem]{Lemma}

\theoremstyle{definition}

\newtheorem{innercustomgeneric}{\customgenericname}
\providecommand{\customgenericname}{}

\newcommand{\newcustomtheorem}[2]{\newenvironment{#1}[1]
  {\renewcommand\customgenericname{#2}
   \renewcommand\theinnercustomgeneric{##1}\innercustomgeneric}{\endinnercustomgeneric}}

\newcustomtheorem{customthm}{Theorem}

\newcommand{\newcustomlemma}[2]{\newenvironment{#1}[1]
  {\renewcommand\customgenericname{#2}
   \renewcommand\theinnercustomgeneric{##1} \innercustomgeneric}{\endinnercustomgeneric}}

\newcustomlemma{customlemma}{Lemma}

\newcustomlemma{customproposition}{Proposition}

\newcommand\relphantom[1]{\mathrel{\phantom{#1}}}

\newcommand{\nn}{\mathbb{N}}

\newcommand{\rd}{\mathbb{R}^d}
\newcommand{\zz}{\mathbb{Z}}

\newcommand{\zd}{\mathbb{Z}^d}

\def\xxi{\vec{\boldsymbol{\xi}}}
\def\|{{\boldsymbol{|}}}

\def\fff{\vec{\boldsymbol{f}}}

\def\xxx{\vec{\boldsymbol{x}}}
\def\yyy{\vec{\boldsymbol{y}}}

\newcommand{\wt}{\widetilde}
\newcommand{\wh}{\widehat}

\numberwithin{equation}{section}

\begin{document}

\begin{thanks}
{The author  is supported in part by NRF grant 2019R1F1A1044075 and by a KIAS Individual Grant MG070001 at Korea Institute for Advanced Study.}
\end{thanks}

\address{School of Mathematics \\
           Korea Institute for Advanced Study, Seoul\\
           Republic of Korea}
   \email{qkrqowns@kias.re.kr}

\author{Bae Jun Park}

\title[multilinear multiplier theorem with endpoint smoothness conditions]{On the failure of multilinear multiplier theorem with endpoint smoothness conditions}

\subjclass[2010]{42B15, 46E35}
\keywords{H\"ormander multiplier theorem, Multilinear operator, Sobolev space with sharp regularity conditions}

\begin{abstract} 
We study a multilinear version of the H\"ormander multiplier theorem, namely
\begin{equation*}
\Vert T_{\sigma}(f_1,\dots,f_n)\Vert_{L^p}\lesssim \sup_{k\in\mathbb{Z}}{\Vert \sigma(2^k\cdot,\dots,2^k\cdot)\widehat{\phi^{(n)}}\Vert_{L^{2}_{(s_1,\dots,s_n)}}}\Vert f_1\Vert_{H^{p_1}}\cdots\Vert f_n\Vert_{H^{p_n}}.
\end{equation*}
We show that the estimate does not hold in the limiting case $\min{(s_1,\dots,s_n)}=d/2$ or $\sum_{k\in J}{({s_k}/{d}-{1}/{p_k})}=-{1}/{2}$  for some $J \subset \{1,\dots,n\}$.
This provides the necessary and sufficient condition on $(s_1,\dots,s_n)$ for the boundedness of $T_{\sigma}$.
\end{abstract}

\maketitle

\section{Introduction}
Let $n$ be a positive integer and $\sigma$ be a bounded function on $(\rd)^n$. The $n$-linear multiplier operator $T_{\sigma}$ associated with $\sigma$ is defined by
\begin{equation*}
T_{\sigma}\big(f_1,\dots,f_n \big)(x):={\int_{(\mathbb{R}^d)^n}{\sigma(\xi_1,\dots,\xi_n)\Big(\prod_{j=1}^{n}\widehat{f_j}(\xi_j)\Big)e^{2\pi i\langle x,\sum_{j=1}^{n}{\xi_j} \rangle}}d\xi_1\cdots d\xi_n}
\end{equation*} 
for Schwartz functions $f_1,\dots,f_n$ in $S(\rd)$, where $\wh{f}(\xi):=\int_{\rd}{f(x)e^{-2\pi i\langle x,\xi\rangle}}dx$ is the Fourier transform of $f$.
The study on $L^{p_1}\times\cdots\times L^{p_n}\to L^p$ boundedness of $T_{\sigma}$, $1/p=1/p_1+\dots+1/p_n$, is one of principal questions in harmonic analysis as a multilinear extension of classical Fourier multiplier theorems and there have been many attempts to characterize $\sigma$ for which the boundedness holds.
A multilinear version of the H\"ormander multiplier theorem was studied by Tomita \cite{Tom} who obtained $L^{p_1}\times\cdots\times L^{p_n}\to L^p$ boundedness $(1<p_1,\dots,p_n,p<\infty)$ under the condition
\begin{equation}\label{hcondition}
\sup_{j\in\zz}{\big\Vert \sigma(2^j\cdot_1,\dots,2^j\cdot_n)\wh{\phi^{(n)}}\big\Vert_{L_s^2((\rd)^n)}}<\infty, \qquad s>nd/2.
\end{equation}
Here $L_s^2((\rd)^n)$ denotes the standard fractional Sobolev space on $(\rd)^n$ and ${\phi^{(n)}}$ is a Schwartz function on $(\rd)^n$ having the properties that
$Supp(\wh{\phi^{(n)}})\subset \{\xxi:=(\xi_1,\dots,\xi_n)\in (\rd)^n: 2^{-1}\leq |\xxi|\leq 2\}$ and $\sum_{k\in\zz}{\wh{\phi^{(n)}}(\xxi/2^k)}=1$ for $\xxi\not= 0$.
Grafakos and Si \cite{Gr_Si} extended the result of Tomita to the case $p\leq 1$, using $L^r$-based Sobolev spaces with $1<r\leq 2$. Fujita and Tomita \cite{Fu_Tom1} provided weighted extensions of these results  for $1<p_1,\dots,p_n,p<\infty$ with 
\begin{equation}\label{newhcondition}
\mathcal{L}_{(s_1,\dots,s_n)}^{2, (n)}[\sigma]:=\sup_{j\in\zz}{\big\Vert \sigma(2^j\cdot_1,\dots,2^j\cdot_n)\wh{\phi^{(n)}}\big\Vert_{L_{(s_1,\dots,s_n)}^2((\rd)^n)}}<\infty, 
\end{equation}
for $s_1,\dots,s_n>d/2$, instead of (\ref{hcondition}), where $L_{(s_1,\dots,s_n)}^{2}((\rd)^n)$ is a product-type Sobolev space with the norm
\begin{equation*}
\Vert f\Vert_{L_{(s_1,\dots,s_n)}^{2}((\rd)^n)}:=\Big(\int_{(\rd)^n}{\Big(\prod_{j=1}^{n}(1+4\pi^2|x_j|^2)^{s_j}\Big)|\wh{f}(x_1,\dots,x_n)|^2}dx_1\cdots dx_n \Big)^{1/2}.
\end{equation*}
The range of $p$ in this result was extended by Grafakos, Miyachi, and Tomita \cite{Gr_Mi_Tom}. 
In \cite{Mi_Tom} Miyachi and Tomita obtained minimal conditions of $s_1,s_2$ in (\ref{newhcondition})
for the $H^{p_1}\times H^{p_2} \to L^p$ boundedness of the bilinear operator $(n=2)$ and more recently, 
Grafakos and Nguyen \cite{Gr_Ng} and Grafakos, Miyachi, Nguyen, and Tomita \cite{Gr_Mi_Ng_Tom} generalize the result of bilinear operators to $n$-linear cases for $n\geq 3$.

\begin{customthm}{A}\label{previousmulti0}
Let $0<p_1,\dots,p_n\leq \infty$, $0<p<\infty$, $1/p_1+\dots+1/p_n=1/p$, and suppose that 
\begin{equation}\label{minimal}
s_1,\dots,s_n>d/2, \qquad \sum_{k\in J}{\Big(\frac{s_k}{d}-\frac{1}{p_k}\Big)}>-\frac{1}{2}
\end{equation} for every nonempty subset $J\subset \{1,2,\dots,n\}$. If $\sigma$ satisfies (\ref{newhcondition}), then we have
\begin{equation}\label{boundresult}
\Vert T_{\sigma}\Vert_{H^{p_1}\times\cdots \times H^{p_n}\to L^p}\lesssim \mathcal{L}_{(s_1,\dots,s_n)}^{2, (n)}[\sigma].
\end{equation}
\end{customthm}
We also refer to \cite{Park3, Park4} for $BMO$ extensions of the result.

The optimality of (\ref{minimal}) was also studied in \cite{Gr_Mi_Ng_Tom, Gr_Ng, Mi_Tom} and indeed, if (\ref{boundresult}) holds, then we must necessarily have 
\begin{equation*}
s_1,\dots,s_n\geq d/2, \qquad \sum_{k\in J}{\Big(\frac{s_k}{d}-\frac{1}{p_k}\Big)}\geq-\frac{1}{2}
\end{equation*} for every nonempty subset $J\subset \{1,2,\dots,n\}$.
The proof is based on a scaling argument by constructing two different multipliers $\sigma^{(\epsilon)}_1$ and $\sigma_2^{(\epsilon)}$ such that for $1\leq m\leq n$,
\begin{equation*}
\mathcal{L}_{(s_1,\dots,s_n)}^{2,(n)}[\sigma_1^{(\epsilon)}]\lesssim \epsilon^{d/2-s_1} ~~ \text{ and }~~ \Vert T_{\sigma_{1}^{(\epsilon)}}\Vert_{H^{p_1}\times\cdots\times H^{p_n}\to L^p}\gtrsim 1 ~~\text{uniformly in } ~\epsilon,
\end{equation*}
\begin{equation*}
\mathcal{L}_{(s_1,\dots,s_n)}^{2,(n)}[\sigma_2^{(\epsilon)}]\lesssim \epsilon^{d/2-(s_1+\dots+s_m)} ~~ \text{ and }~~ \Vert T_{\sigma_{2}^{(\epsilon)}}\Vert_{H^{p_1}\times\cdots\times H^{p_n}\to L^p}\gtrsim \epsilon^{d-d(1/p_1+\dots+1/p_m)}.
\end{equation*}

However, (\ref{boundresult}) still remains unknown in the critical case 
\begin{equation*}
\min{(s_1,\dots,s_n)}=d/2 \quad \text{ or }\quad \sum_{k\in J}{\Big(\frac{s_k}{d}-\frac{1}{p_k}\Big)}=-\frac{1}{2} \quad \text{for some }~J\subset \{1,\dots,n\}.
\end{equation*}
  In this paper we shall address this question so that the necessary and sufficient conditions on $(s_1,\dots,s_n)$ for (\ref{boundresult}) are completely achieved.
  
Our main result is the following.
\begin{theorem}\label{main}
Let $0<p_1,\dots,p_n\leq \infty$, $0<p<\infty$, $1/p_1+\dots+1/p_n=1/p$. Suppose $(\ref{newhcondition})$ holds for $s_1,\dots,s_n>0$.
Then (\ref{boundresult}) does not hold if
\begin{equation}\label{ffcondition}
\min{(s_1,\dots,s_n)}\leq d/2
\end{equation} or
\begin{equation*}
\sum_{k\in J}{\Big(\frac{s_k}{d}-\frac{1}{p_k}\Big)}\leq-\frac{1}{2} \quad \text{ for some }~J\subset \{1,\dots,n\}.
\end{equation*} 
\end{theorem}

In order to prove Theorem \ref{main}, two different multipliers will be constructed. For the first case (\ref{ffcondition}) we will use a more sophisticated scaling argument.
The proof of the other case relies on a variant of Bessel potentials that appeared in \cite{Gr_Park, Park2}.

\section{Proof of Theorem \ref{main}}

For notational convenience, we will occasionally write
$\fff:=(f_1,\dots,f_n)$, $\xxi:=(\xi_1,\dots,\xi_n)$, $\xxx:=(x_1,\dots,x_n)$, $\yyy:=(y_1,\dots,y_n)$, $d\xxx:=dx_1\cdots dx_n$, $d\yyy:=dy_1\cdots dy_n$.

\begin{lemma}\label{katoponce}
Let $g$ be a function in $L_{(s_1,\dots,s_n)}^{2}((\rd)^n)$. Then
\begin{equation*}
\big\Vert \wh{\phi^{(n)}} \cdot g \big\Vert_{L_{(s_1,\dots,s_n)}^{2}((\rd)^n)}\lesssim \Vert g \Vert_{L_{(s_1,\dots,s_n)}^{2}((\rd)^n)}.
\end{equation*}

\end{lemma}
\begin{proof}
By using the Minkowski inequality, we obtain that
\begin{align*}
&\big\Vert \wh{\phi^{(n)}} \cdot g\big\Vert_{L_{(s_1,\dots,s_n)}^{2}((\rd)^n)}=\Big( \int_{(\rd)^n}{\Big( \prod_{j=1}^{n}\big(1+4\pi^2|x_j|^2\big)^{s_j}\Big)\big|\phi^{(n)}\ast g^{\vee}(\xxx) \big|^2}d\xxx\Big)^{1/2}\\
&\leq \int_{(\rd)^n}{\phi^{(n)}(\yyy)\Big(\int_{(\rd)^n}{ \Big( \prod_{j=1}^{n}\big(1+4\pi^2|x_j|^2\big)^{s_j}\Big)\big| g^{\vee}(\xxx-\yyy)\big|^2 }d\xxx \Big)^{1/2}}d\yyy
\end{align*}
and then using $\big(1+4\pi^2|x_j|^2\big)^{s_j}\lesssim \big(1+4\pi^2|x_j-y_j|^2\big)^{s_j}\big(1+4\pi^2|y_j|^2\big)^{s_j}$, the last expression is dominated by a constant multiple of $\Vert g \Vert_{L_{(s_1,\dots,s_n)}^{2}((\rd)^n)}$.

\end{proof}

{\bf Case 1 :} Suppose that  $\min{(s_1,\dots,s_n)}\leq d/2$.
 
Let $\{\lambda_1,\lambda_2,\dots\}$ be a sequence of disjoint lattices in $\mathbb{Z}^d$ such that 
\begin{eqnarray}\label{secondition}
|\lambda_m|\leq \sqrt{d}m^{1/d}.
\end{eqnarray} 
One way to select such a sequence is as follows.
For each $k\in\nn$ let $$\lambda_{k^d}:=(k,0,0,\dots,0)\in \zd.$$
We observe that there are at most $d(k+1)^{d-1}$ integers between $k^d$ and $(k+1)^{d}$, and
there exist at least $2d(2k-1)^{d-1}$ lattices on the surface of cube $[-k,k]^d$.
Since $d(k+1)^{d-1}\leq 2d(2k-1)^{d-1}$ we can choose lattices $\lambda_{k^d+1},\lambda_{k^d+2},\dots,\lambda_{(k+1)^d-1}$ on the surface of the cube and then clearly the length of those lattices is less than $\sqrt{d}k$, which yields (\ref{secondition}).
 
 It is enough to consider the case  $s_1\leq s_2,\dots,s_n$ and $s_1\leq d/2$ as other cases will follow from a rearrangement.

Let $\eta$ and $\wt{\eta}$ denote Schwartz functions on $\rd$ having the properties that $\eta\geq 0$, $\eta(x)\geq c$ on $\{x\in\rd:|x|\leq \frac{1}{100}\}$ for some $c>0$, $Supp(\wh{\eta})\subset \{\xi\in\rd: |\xi|\leq \frac{1}{1000}\}$, $\wh{\wt{\eta}}(\xi)=1$ for $|\xi|\leq \frac{1}{1000}$, and $Supp(\wh{\wt{\eta}})\subset \{\xi\in\rd: |\xi|\leq \frac{1}{100}\}$. 
We further denote by $\theta$ and $\wt{\theta}$ two Schwartz functions on $\rd$ such that $Supp(\wh{\theta})\subset \big\{\xi\in\rd: \frac{1}{2000\sqrt{n}}\leq |\xi|\leq \frac{1}{1000\sqrt{n}}\big\},$ $Supp(\wh{\wt{\theta}})\subset \big\{\xi\in\rd:  |\xi|\leq \frac{1}{100\sqrt{n}}\big\},$ and $\wh{\wt{\theta}}(\xi)=1$  for $ |\xi|\leq \frac{1}{1000\sqrt{n}}.$
Then it is clear that  $\eta\ast \wt{\eta}=\eta$ and $\theta\ast \wt{\theta}=\theta$. Let $e_1:=(1,0,\dots,0)\in \zd$.

Let $1/2<\delta\leq 1$ and set $N>100$ to be a sufficiently large number.
We define
\begin{equation*}
\sigma^{(N)}(\xxi):=\sum_{m=100}^{N}{\frac{1}{m}\frac{1}{(\ln{m})^{\delta}}e^{2\pi i\langle \lambda_m,\xi_1-e_1\rangle}\wh{\wt{\eta}}(\xi_1-e_1)\wh{\wt{\theta}}(\xi_2)\cdots\wh{\wt{\theta}}(\xi_n)}.
\end{equation*}
Then it follows from the supports of $\wh{\wt{\eta}}$ and $\wh{\wt{\theta}}$ that $\sigma^{(N)}$ is supported in $\{\xxi\in (\rd)^n: \frac{99}{100} \leq |\xxi| \leq\frac{51}{50}\}$, which implies that $\sigma^{(N)}(2^j\xxi)\wh{\phi^{(n)}}(\xxi)$ vanishes unless $-1\leq j\leq 1$.
Moreover, due to Lemma \ref{katoponce} we have
\begin{equation*}
\mathcal{L}_{(s_1,\dots,s_n)}^{2,(n)}[\sigma^{(N)}]\lesssim \max_{-1\leq j\leq 1}{\big\Vert \sigma^{(N)}(2^j\cdot_1,\dots,2^j\cdot_n)\big\Vert_{L_{(s_1,\dots,s_n)}^{2}((\rd)^n)}},
\end{equation*} which is further estimated, using scaling, by a constant times
\begin{align*}
&\Vert \sigma^{(N)}\Vert_{L_{(s_1,\dots,s_n)}^2((\rd)^n)}\\
&\lesssim \Big\Vert \sum_{m=100}^{N}{\frac{1}{m}\frac{1}{(\ln{m})^{\delta}}e^{2\pi i\langle \lambda_m,\cdot-e_1\rangle}\wh{\wt{\eta}}(\cdot-e_1)}\Big\Vert_{L_{s_1}^2(\rd)}\\
&=\Big( \int_{\rd}{\big( 1+4\pi^2|x|^2\big)^{s_1}\Big(\sum_{m=100}^{N}{\frac{1}{m}\frac{1}{(\ln{m})^{\delta}}\big|\wt{\eta}(\lambda_m-x)\big|} \Big)^{2}}dx\Big)^{1/2}\\
&\lesssim_M\Big( \int_{\rd}{\big( 1+4\pi^2|x|^2\big)^{s_1}\Big(\sum_{m=100}^{\infty}{\frac{1}{m}\frac{1}{(\ln{m})^{\delta}}\frac{1}{(1+|x-\lambda_m|)^M}} \Big)^{2}}dx\Big)^{1/2}
\end{align*} for sufficiently large $M>0$.

Let \begin{equation*}
E_0:=\big\{x\in\rd: |x|\leq 1 \big\},
\end{equation*}
\begin{equation*}
E_j:=\big\{x\in\rd: 2^{j-1}<|x|\leq 2^j \big\}, \qquad j\geq 1.
\end{equation*}
 Then the preceding expression is less than a constant times that
 \begin{equation*}
 \Big(\sum_{j=0}^{\infty}{2^{2js_1}\int_{E_j}{\Big(\sum_{m=100}^{\infty}{\frac{1}{m}\frac{1}{(\ln{m})^{\delta}}\frac{1}{(1+|x-\lambda_m|)^M}} \Big)^{2}}dx} \Big)^{1/2}
 \end{equation*}
and the sum over $m$ can be written as 
\begin{equation*}
\sum_{m\geq 100: 2^j\geq 10\sqrt{d}m^{1/d}}{\cdots} + \sum_{m:2^j< 10\sqrt{d}m^{1/d}}{\cdots} =:\mathcal{I}^{(j)}+\mathcal{J}^{(j)}.
\end{equation*}
Choose $M>d$ and we see that for $x\in E_j$,
\begin{equation*}
\mathcal{I}^{(j)}\lesssim 2^{-jM}(1+j)^{1-\delta}\lesssim 2^{-jd}(1+j)^{-\delta}
\end{equation*}
and \begin{equation*}
\mathcal{J}^{(j)}\lesssim 2^{-jd}(1+j)^{-\delta}\sum_{m=100}^{\infty}{\frac{1}{(1+|x-\lambda_m|)^M}}\lesssim 2^{-jd}(1+j)^{-\delta}
\end{equation*}
where the last inequality holds since $\lambda_m$'s are disjoint lattices in $\zd$.

Combining these estimates, we conclude that
\begin{align}\label{firstest}
\mathcal{L}_{(s_1,\dots,s_n)}^{2,(n)}[\sigma^{(N)}]&\lesssim \Big( \sum_{j=0}^{\infty}{2^{-j(d-2s_1)}(1+j)^{-2\delta}}\Big)^{1/2}\nonumber\\
&\leq \Big(\sum_{j=0}^{\infty}{(1+j)^{-2\delta}} \Big)^{1/2}\lesssim 1,\qquad \text{ uniformly in }~ N
\end{align} since $|E_j|\approx 2^{jd}$, $s_1\leq d/2$, and $2\delta>1$.

On the other hand, for $0<\epsilon<1/100$, let
\begin{equation*}
f^{(\epsilon)}_1(x):=\epsilon^{d/p_1}\eta(\epsilon x)e^{2\pi i\langle x,e_1\rangle}, \qquad f^{(\epsilon)}_j(x)::=\epsilon^{d/p_j}{\theta}(\epsilon x), ~~j=2,\dots,n
\end{equation*} 
Then it is clear that for each $1\leq j\leq n$
\begin{equation}\label{secondest}
\big\Vert f_j^{(\epsilon)}\big\Vert_{H^{p_j}(\rd)}\lesssim 1,\qquad \text{ uniformly in }~~\epsilon.
\end{equation}

We now assume toward a contradiction that for $\fff^{(\epsilon)}:=(f_1^{(\epsilon)},\dots,f_n^{(\epsilon)})$
\begin{equation*}
\big\Vert T_{\sigma^{(N)}}\fff^{(\epsilon)}\big\Vert_{L^p(\rd)}\lesssim \sup_{j\in\zz}{\big\Vert \sigma^{(N)}(2^j\cdot)\wh{\phi^{(n)}}\big\Vert_{L_{(s_1,\dots,s_n)}^{2,(n)}((\rd)^n)}}\prod_{j=1}^{n}{\Vert f_j^{(\epsilon)}\Vert_{H^{p_j}(\rd)}}.
\end{equation*}
Then the estimates (\ref{firstest}) and (\ref{secondest}) yield that
\begin{equation}\label{contradiction}
\big\Vert T_{\sigma^{(N)}}\fff^{(\epsilon)}\big\Vert_{L^p(\rd)}\lesssim 1, \qquad \text{ uniformly in }~~ \epsilon,~ N.
\end{equation}

Note that
\begin{equation*}
T_{\sigma^{(N)}}\fff^{(\epsilon)}(x)=\epsilon^{d/p}\big( \theta(\epsilon x )\big)^{n-1}\sum_{m=100}^{N}{\frac{1}{m}\frac{1}{(\ln{m})^{\delta}}\eta(\epsilon x+\epsilon \lambda_m)e^{2\pi i\langle x,e_1\rangle}}
\end{equation*}
since $\wh{\eta}( \xi/\epsilon)\wh{\wt{\eta}}(\xi)=\wh{\eta}( \xi/\epsilon)$, $\wh{\theta}( \xi/\epsilon)\wh{\wt{\theta}}(\xi)=\wh{\theta}( \xi/\epsilon)$ , and $1/p=1/p_1+\dots+1/p_n$.
 Thus we have 
\begin{align*}
\big\Vert T_{\sigma^{(N)}}\fff^{(\epsilon)}\big\Vert_{L^p(\rd)}=\Big\Vert |\theta|^{n-1}\sum_{m=100}^{N}{\frac{1}{m}\frac{1}{(\ln{m})^{\delta}}\eta(\cdot+\epsilon\lambda_m)}\Big\Vert_{L^p(\rd)}
\end{align*} by using a scaling argument.
Then (\ref{contradiction}) and the Fatou lemma yield that
\begin{align*}
1&\gtrsim \liminf_{\epsilon\to 0}{\big\Vert T_{\sigma^{(N)}}\fff^{(\epsilon)}\big\Vert_{L^p(\rd)}}\\
 &\geq \Big\Vert \liminf_{\epsilon\to 0} |\theta|^{n-1}\sum_{m=100}^{N}{\frac{1}{m}\frac{1}{(\ln{m})^{\delta}}\eta(\cdot+\epsilon\lambda_m)}\Big\Vert_{L^p(\rd)}\\
 &=\big\Vert |\eta| |\theta|^{n-1}\big\Vert_{L^p(\rd)}\sum_{m=100}^{N}{\frac{1}{m}\frac{1}{(\ln{m})^{\delta}}}\approx \sum_{m=100}^{N}{\frac{1}{m}\frac{1}{(\ln{m})^{\delta}}}.
\end{align*} 
Since the above estimate is independent of $N>100$, by taking $N\to \infty$, we see that
\begin{equation*}
1\gtrsim \sum_{m=100}^{N}{\frac{1}{m}\frac{1}{(\ln{m})^{\delta}}}\to \infty,
\end{equation*}
which finally leads to a contradiction and the desired result follows.\\

{\bf Case 2 :} Suppose that $s_1,\dots,s_n>d/2$ and $\sum_{k\in J}{\big({s_k}/{d}-{1}/{p_k}\big)}\leq-{1}/{2}$ for some $J\subset \{1,\dots,n\}$. 
Without loss of generality, we may assume $J=\{1,\dots,m\}$ for some $1\leq m\leq n$, by using a rearrangement argument.

We first consider the case $1\leq m<n$.
Then the condition 
\begin{equation*}
\sum_{k=1}^{m}{\big({s_k}/{d}-{1}/{p_k}\big)}\leq-{1}/{2}
\end{equation*} is equivalent to
\begin{equation}\label{equivcondition}
s_1+\dots +s_m+d/2 \leq d/p_1+\dots+d/p_m=d/p-\big(d/p_{m+1}+\dots+d/p_n).
\end{equation}
On the other hand, we observe from $s_j>d/2$, $1\leq j\leq n$, that
\begin{equation*}
s_1+\dots+s_m+d/2>(m+1)d/2,
\end{equation*} which further implies that
\begin{equation*}
m+1<2/p-\big(2/p_{m+1}+\dots+2/p_n \big).
\end{equation*}
Now we choose $\tau, \tau_{m+1},\dots, \tau_n>0$ such that
\begin{equation*}
\tau_{m+1}>2/p_{m+1},\dots,\tau_n>2/p_n
\end{equation*}
and
\begin{equation}\label{taucondition}
1<\tau<m+1<2/p-(\tau_{m+1}+\dots+\tau_n)<2/p-(2/p_{m+1}+\dots+2/p_n).
\end{equation}

Let $\psi,\wt{\psi}\in S(\rd)$ satisfy $\psi\geq 0$, $\psi(0)\not= 0$, $Supp(\wh{\psi})\subset \{\xi\in\rd: |\xi|\leq \frac{1}{200mn}\}$, $Supp(\wh{\wt{\psi}})\subset \{\xi\in\rd: |\xi|\leq \frac{1}{100n}\}$, and $\wh{\wt{\psi}}(\xi)=1$ for $|\xi|\leq \frac{1}{200n}$.  

We define
\begin{equation*}
\mathcal{H}^{(m)}(x):=\frac{1}{(1+4\pi^2|x|^2)^{\frac{1}{2}(s_1+\dots+s_m+\frac{d}{2})}}\frac{1}{(1+\ln{(1+4\pi^2|x|^2)})^{\frac{\tau}{2}}}, \quad x\in\rd,
\end{equation*}
\begin{equation*}
K^{(m)}(x):=\mathcal{H}^{(m)}\ast \psi(x), \quad x\in\rd,
\end{equation*}
and
\begin{equation*}
M^{(m)}(\xi_1,\dots,\xi_m):=\wh{K^{(m)}}\Big(\frac{1}{m}\sum_{l=1}^{m}{(\xi_l-\mu_1)} \Big)  \prod_{j=2}^{m} \wh{\psi}\Big(\frac{1}{m}\sum_{l=1}^{m}{(\xi_l-\xi_j)} \Big)
\end{equation*}
where $\mu_1:=(n^{-1/2},0,\dots,0)\in\mathbb{R}^d$. Here, $M^{(m)}$ is defined on $(\rd)^m$.
Then the multiplier $\sigma$ on $(\rd)^n$ is defined by
\begin{equation*}
\sigma(\xi_1,\dots,\xi_n):=M^{(m)}(\xi_1,\dots,\xi_m)\wh{\wt{\psi}}(\xi_{m+1}-\mu_1)\cdots \wh{\wt{\psi}}(\xi_{n}-\mu_1).
\end{equation*}
To investigate the support of $\sigma$ we first look at the support of $M^{(m)}$.
From the support $\wh{\psi}$, we have
\begin{equation*}
\big| \xi_1+\dots+\xi_m-m\mu_1\big|\leq \frac{1}{200n},
\end{equation*} and for each $2\leq j\leq m$
\begin{equation}\label{jest}
\big| \xi_1+\dots+\xi_m-m\xi_j\big|\leq \frac{1}{200n}.
\end{equation}
By adding up all of them, we obtain 
\begin{equation}\label{1est}
\big| \xi_1-\mu_1\big|\leq \frac{1}{200n}
\end{equation}
and the sum of (\ref{jest}) and (\ref{1est}) yields that for each $2\leq j\leq m$
\begin{equation*}
\big| \mu_1+\xi_2+\dots+\xi_m-m\xi_j\big|\leq \frac{1}{100n}.
\end{equation*}
Let us call the above estimate $\mathcal{E}(j)$.
Then for  $2\leq j\leq m$, it follows from  $$\mathcal{E}(j)+\sum_{l=2}^{m}\mathcal{E}(l)$$
that 
\begin{equation*}
\big|\xi_j-\mu_1\big|\leq \frac{1}{100n},
\end{equation*} which proves, together with (\ref{1est}),
\begin{align}\label{supportm}
 Supp(M^{(m)})&\subset \big\{ (\xi_1,\dots,\xi_m)\in (\rd)^m:   |\xi_j-\mu_1|\leq \frac{1}{100n}, ~ 1\leq j\leq m \big\}.
\end{align}
Since $\wh{\wt{\psi}}$ is also supported in $\{\xi\in\rd: |\xi|\leq \frac{1}{100{n}}\}$, it is clear that 
\begin{align*}
Supp(\sigma)&\subset  \big\{ (\xi_1,\dots,\xi_n)\in (\rd)^n:   |\xi_j-\mu_1|\leq \frac{1}{100n}, ~ 1\leq j\leq n \big\}\\
&\subset \big\{\xxi:=(\xi_1,\dots,\xi_n)\in (\rd)^n:    \frac{99}{100}\leq |\xxi|\leq \frac{101}{100}     \big\},
\end{align*}  which implies that $\sigma(2^j\xxi )\wh{\phi^{(n)}}(\xxi)$ vanishes unless $-1\leq j\leq 1$.
Furthermore, using Lemma \ref{katoponce} and the scaling argument we used in {\bf Case 1},
\begin{equation*}
\mathcal{L}_{(s_1,\dots,s_n)}^{2,(n)}[\sigma]\lesssim \Vert \sigma\Vert_{L_{(s_1,\dots,s_n)}^{2}((\rd)^n)}.
\end{equation*}

We notice that $\Vert \sigma\Vert_{L_{(s_1,\dots,s_n)}^{2}((\rd)^n)}$ is dominated by a constant times
\begin{align}\label{sigmaest}
& \big\Vert M^{(m)}\big\Vert_{L_{(s_1,\dots,s_m)}^{2}((\rd)^m)}\prod_{j=m+1}^{n}{\big\Vert \wh{\wt{\psi}}\big\Vert_{L_{s_j}^{2}(\rd)}}\lesssim \big\Vert M^{(m)}\big\Vert_{L_{(s_1,\dots,s_m)}^{2}((\rd)^m)}\nonumber\\
&=\Big( \int_{(\rd)^m}{\Big( \prod_{j=1}^{m}{(1+4\pi^2|x_j|^2)^{s_j}}\Big)\big|(M^{(m)})^{\vee}(x_1,\dots,x_m)\big|^2}dx_1\dots dx_m\Big)^{1/2}.
\end{align}
By using a change of variables
\begin{equation*}
\zeta_1:=\frac{1}{m}\sum_{l=1}^{m}{(\xi_l-\mu_1)}, \qquad \text{ and }\qquad \zeta_j:=\frac{1}{m}\sum_{l=1}^{m}{(\xi_l-\xi_j)}, \quad 2\leq j\leq m,
\end{equation*}
so that
\begin{equation}\label{system}
\xi_1=\zeta_1+\dots+\zeta_m+\mu_1, \qquad \text{ and }\qquad \xi_j=\zeta_1-\zeta_j+\mu_1, \quad 2\leq j\leq m,
\end{equation}
we can write
\begin{equation}\label{fmest}
\big|(M^{(m)})^{\vee}(x_1,\dots,x_m)\big|=m\big|K^{(m)}(x_1+\dots+x_m) \big|\prod_{j=2}^{m}{\big|\psi(x_1-x_j) \big|}
\end{equation} since the Jacobian of the system (\ref{system}) is $m$.

Consequently, (\ref{sigmaest}) is less than a constant multiple of 
\begin{equation*}
\Big(\int_{(\rd)^m}{\Big( \prod_{j=1}^{m}{(1+4\pi^2|x_j|^2)^{s_j}}\Big)\big|K^{(m)}(x_1+\dots+x_m) \big| \big|\prod_{j=2}^{m}{\big|\psi(x_1-x_j) \big| }}dx_1\dots dx_m \Big)^{1/2}
\end{equation*} and we perform another change of variables 
\begin{equation*}
y_1:=x_1+\dots+x_m, \qquad \text{ and }\qquad y_j:=x_1-x_j, \quad 2\leq j\leq m,
\end{equation*}
which is equivalent to 
\begin{equation*}
x_1=\frac{1}{m}\sum_{l=1}^{m}{y_l}, \qquad \text{ and }\qquad  x_j=\frac{1}{m}\sum_{l=1}^{m}{(y_l-y_j)}, \quad 2\leq j\leq m,
\end{equation*}
to obtain that the preceding expression is controlled by a constant times
\begin{align*}
&\Big(\int_{\rd}{\big(1+4\pi^2|y_1|^2\big)^{s_1+\dots+s_m}\big|K^{(m)}(y_1) \big|^2}dy_1 \Big)^{1/2}\\
&\relphantom{=}\times \prod_{j=2}^{m}{\Big( \int_{\rd}{\big(1+4\pi^2|y_j|^2 \big)^{s_1+\dots+s_m}\big|\psi(y_j) \big|^2}dy_j\Big)^{1/2}}\\
& \lesssim \Big(\int_{\rd}{\big(1+4\pi^2|y|^2\big)^{s_1+\dots+s_m}\big|\mathcal{H}^{(m)}\ast \psi(y) \big|^2}dy \Big)^{1/2}
\end{align*}
where we applied the triangle inequality
\begin{equation*}
\big(1+4\pi^2|x_j|^2 \big)^{s_j}\lesssim \big( 1+4\pi^2|y_1|^2\big)^{s_j}\cdots\big( 1+4\pi^2|y_m|^2\big)^{s_j} \qquad \text{ for each }~ 1\leq j\leq m. 
\end{equation*}
Since $\mathcal{H}^{(m)}\ast \psi(y)\lesssim \mathcal{H}^{(m)}(y)$, we conclude that
\begin{align*}
\mathcal{L}_{(s_1,\dots,s_n)}^{2,(n)}[\sigma]&\lesssim \Big( \int_{\rd}{\big( 1+4\pi^2|y|^2\big)^{s_1+\dots+s_m}\big| \mathcal{H}^{(m)}(y)\big|^2}dy\Big)^{1/2}\\
&=\Big( \int_{\rd}{\frac{1}{(1+4\pi^2|y|^2)^{d/2}}\frac{1}{(1+\ln{(1+4\pi^2|y|^2)})^{\tau}}}dy\Big)^{1/2}\lesssim 1
\end{align*} since $\tau>1$.

To achieve 
\begin{equation}\label{achieve}
\Vert T_{\sigma}\Vert_{H^{p_1}\times \dots\times H^{p_n}\to L^p}=\infty,
\end{equation}
let
\begin{equation*}
h^{(j)}(x):=\frac{1}{(1+4\pi^2|x|^2)^{\frac{1}{2}\frac{d}{p_j}}}\frac{1}{(1+\ln{(1+4\pi^2|x|^2)})^{\frac{\tau_j}{2}}}, \qquad m+1\leq j\leq n
\end{equation*}
and we define 
\begin{equation*}
f_1(x)=\dots=f_m(x)=2^d\wt{\psi}(2x)e^{2\pi i\langle x,\mu_1\rangle},
\end{equation*}
\begin{equation*}
f_j(x):=h^{(j)}\ast \psi(x)e^{2\pi i\langle x,\mu_1\rangle}, \quad m+1\leq 1\leq n.
\end{equation*}
Clearly, $\Vert f_j\Vert_{H^{p_j}(\rd)}\lesssim 1$ for $1\leq j\leq m$ and
\begin{equation*}
\Vert f_j\Vert_{H^{p_j}}\approx \Vert f_j\Vert_{L^{p_j}(\rd)}\lesssim \big\Vert h^{(j)}\big\Vert_{L^{p_j}(\rd)}\lesssim 1 \quad m+1\leq j\leq n
\end{equation*} since $\tau_jp_j>2$.

On the other hand, using (\ref{supportm}) and the facts that $\psi\ast \wt{\psi}=\psi$ and
\begin{equation*}
\wh{f_j}(\xi)=1 \qquad \text{ for }~~ |\xi-\mu_1|\leq \frac{1}{100n} ~ \text{ and }~  ~ 1\leq j\leq m,
\end{equation*}
we see that
\begin{equation*}
\sigma(\xi_1,\dots,\xi_n)\wh{f_1}(\xi_1)\cdots\wh{f_n}(\xi_n)=M^{(m)}(\xi_1,\dots,\xi_m)\wh{f_{m+1}}(\xi_{m+1})\cdots\wh{f_n}(\xi_{n}),
\end{equation*} which implies that
\begin{equation*}
T_{\sigma}\fff (x)= \big( M^{(m)}\big)^{\vee}(x,\dots,x)f_{m+1}(x)\cdots f_{n}(x).
\end{equation*}
Since $\mathcal{H}^{(m)}$, $h^{(j)}$, and $\psi$ are nonnegative functions,
\begin{equation*}
\big|T_{\sigma}\fff(x)\big|=m \mathcal{H}^{(m)}\ast \psi(mx)\psi(0)^{m-1}\prod_{j=m+1}^{n}{\big( h^{(j)}\ast\psi(x)\big)}
\end{equation*}
where (\ref{fmest}) is applied. Now, using the fact that 
\begin{equation*}
\mathcal{H}^{(m)}(x-y)\geq \mathcal{H}^{(m)}(x)\mathcal{H}^{(m)}(y),
\end{equation*} 
\begin{equation*}
{h}^{(j)}(x-y)\geq {h}^{(j)}(x){h}^{(j)}(y),\qquad m+1\leq j\leq n,
\end{equation*} 
we obtain that
\begin{align*}
&\big\Vert T_{\sigma}\fff(x)\big\Vert_{L^p(\rd)}\gtrsim \Big\Vert \mathcal{H}^{(m)}(m\cdot)\prod_{j=m+1}^{n}{h^{(j)}}\Big\Vert_{L^p(\rd)}\\
&\gtrsim \Big( \int_{\rd}{\frac{1}{(1+4\pi^2|x|^2)^{d/2}}\frac{1}{(1+\ln{(1+4\pi^2|x|^2)})^{\frac{p}{2}(\tau+\tau_{m+1}+\dots+\tau_n)}}}dx\Big)^{1/p}
\end{align*}
where the second inequality is due to (\ref{equivcondition}).
Since $\tau+\tau_{m+1}+\dots+\tau_n<2/p$, which follows from (\ref{taucondition}), the preceding expression diverges and this completes the proof of (\ref{achieve}).

When $m=n$, the exactly same argument is applicable with $1<\tau<n+1<\frac{2}{p}$, $\sigma:=M^{(n)}$, and $f_j(x):=2^d\wt{\psi}(2x)e^{2\pi i\langle x,\mu_1\rangle}$ for $1\leq j\leq n$. Since the proof is just a repetition, we omit the details.

This ends the proof.

\section*{Acknowledgement}

{Part of this research was carried out during my stay at the University of Missouri-Columbia. I would like to thank L. Grafakos for his invitation, hospitality, and very useful discussions during the stay. I also would like to express gratitude to the anonymous referee for the careful reading and suggestions.  }


\begin{thebibliography}{99}



\bibitem{Fu_Tom1}
M. Fugita and N. Tomita,  \emph{Weighted norm inequalities for multilinear Fourier multipliers}, Trans. Amer. Math. Soc. \textbf{364} (2012) 6335-6353.


\bibitem{Gr_Mi_Tom}
L. Grafakos, A. Miyachi, and N. Tomita,  \emph{On multilinear Fourier multipliers of limited smoothness}, Can. J. Math. \textbf{65} (2013) 299-330.

\bibitem{Gr_Mi_Ng_Tom}
L. Grafakos, A. Miyachi, H.V. Nguyen, and N. Tomita,  \emph{Multilinear Fourier multipliers with minimal Sobolev regularity, II}, J. Math. Soc. Japan \textbf{69} (2017) 529-562.

\bibitem{Gr_Ng}
L. Grafakos and H.V. Nguyen,  \emph{Multilinear Fourier multipliers with minimal Sobolev regularity, I}, Colloquium Math. \textbf{144} (2016) 1-30.


\bibitem{Gr_Park}
L. Grafakos and B. Park,  \emph{Sharp Hardy space estimates for multipliers}, submitted.


\bibitem{Gr_Si}
L. Grafakos and Z. Si,  \emph{The H\"ormander multiplier theorem for multilinear operators}, J. Reine Angew. Math. \textbf{668} (2012) 133-147.



\bibitem{Mi_Tom}
A. Miyachi and N. Tomita, \emph{Minimal smoothness conditions for bilinear Fourier multipliers}, Rev. Mat. Iberoam. \textbf{29} (2013) 495-530.



\bibitem{Park2}
B. Park,  \emph{Fourier multiplier theorems for Triebel-Lizorkin spaces}, Math. Z  \textbf{293} (2019) 221-258.

\bibitem{Park3}
B. Park,  \emph{Equivalence of (quasi-)norms on a vector-valued function space and its applications to multilinear operators}, Indiana Univ. Math. J, to appear.

\bibitem{Park4}
B. Park,  \emph{BMO multilinear multiplier theorem of Mikhlin-Hormander type}, submitted.



\bibitem{Tom}
N. Tomita,  \emph{A H\"ormander type multiplier theorem for multilinear operators}, J. Func. Anal.  \textbf{259}
(2010) 2028-2044.






\end{thebibliography}
\end{document}